\documentclass[11pt]{article}

\usepackage{amsmath,amsthm}
\usepackage{amsfonts,amsrefs}
\usepackage{verbatim}
\usepackage{amssymb}
\usepackage{txfonts}
\usepackage{amscd}
\usepackage{bbm}
\usepackage{hyperref}
\usepackage{mathrsfs}
\usepackage{fouriernc}

\usepackage{tikz-cd}

\newcommand{\set}[1]{\,\left\{#1\right\}}

\newtheorem{theorem}{Theorem}
\newtheorem{corollary}[theorem]{Corollary}
\newtheorem{lemma}[theorem]{Lemma}
\newtheorem{proposition}[theorem]{Proposition}
\newtheorem{definition}[theorem]{Definition}
\newtheorem{question}[theorem]{Question}
\newtheorem{remark}[theorem]{Remark}

\newcommand{\HH}{\mathcal{H}}

\newcommand{\Addresses}{{
  \bigskip
  \footnotesize

K. Li, P. W. Nowak and S. Pooya, \textsc{Institute of Mathematics of the Polish Academy of Sciences, \'{S}niadeckich 8, 00-656 Warszawa, Poland}\\
{\texttt{kli@impan.pl}}, {\texttt{pnowak@impan.pl}}, {\texttt{spooya@impan.pl}}
\par\nopagebreak
}}

\usepackage{authblk}

\begin{document}

\title{Higher Kazhdan projections, $\ell_2$-Betti numbers and Baum-Connes conjectures}
\author{Kang Li}
\author{Piotr W. Nowak}
\author{ Sanaz Pooya}

\affil{ }

\maketitle

\begin{abstract}
We introduce higher-dimensional analogs of Kazhdan projections in matrix algebras over group $C^*$-algebras 
and Roe algebras. These projections are constructed  in the framework of cohomology with coefficients in unitary representations and
in certain cases give rise to non-trivial $K$-theory classes.
We apply the higher Kazhdan projections to establish a relation between 
$\ell_2$-Betti numbers of a group and surjectivity of different Baum-Connes type assembly maps. 
\end{abstract}

Kazhdan projections are certain idempotents whose existence in the maximal group $C^*$-algebra $C^*_{\max}(G)$ of a group $G$ characterizes Kazhdan's property $(T)$
for $G$. They have found an important application in higher index theory, as the non-zero $K$-theory class represented by a Kazhdan projection in $K_0(C^*_{\max}(G))$ obstructs the Dirac-dual Dirac method of proving the Baum-Connes conjecture. 
This idea has been used effectively in the coarse setting \cite{higson, higson-lafforgue-skandalis}, where Kazhdan projections were used to construct 
counterexamples to the coarse Baum-Connes conjecture and to the Baum-Connes conjecture with coefficients. 
However, Kazhdan projections related to property $(T)$ are not applicable to the classical Baum-Connes conjecture, as the image of a Kazhdan projection 
in the reduced group $C^*$-algebra vanishes for any infinite group.

The goal of this work is to introduce 
higher-dimensional analogs of Kazhdan projections with the motivation of applying them 
to the $K$-theory of group $C^*$-algebra and to higher index theory.
These higher Kazhdan projections are defined in the context of higher cohomology with coefficients in unitary representations 
and they are elements of matrix algebras over group $C^*$-algebras. Their defining feature is that their 
image in certain unitary representations is the orthogonal projection onto the harmonic $n$-cochains. 
We  show that the higher Kazhdan projections exist in many natural cases and that they are non-zero. 

We then investigate the classes represented by higher Kazhdan projection  in the $K$-theory of  group $C^*$-algebras. 
This is done by establishing a connection between the properties of the introduced projections and $\ell_2$-Betti numbers of the 
group. In particular, we can conclude that the higher Kazhdan projections over the reduced group $C^*$-algebra can be non-zero  
and can give rise to non-zero $K$-theory 
classes in $K_0(C^*_r(G))$.

To present the connections between the existence of higher Kazhdan projections and $\ell_2$-invariants recall that given 
a finitely generated group $G$  the subring $\Lambda^G\subseteq \mathbb{Q}$ is generated 
from $\mathbb{Z}$ by adjoining inverses of all orders of finite subgroups, as in \cite[Theorem 0.3]{luck}.
The cohomological Laplacian in degree $n$ can be represented by a matrix with coefficients in the group ring $\mathbb{R}G$, denoted $\Delta_n$.
In the case of the regular representation $\lambda$ and the reduced group $C^*$-algebra $C^*_r(G)$ 
the operator $\Delta_n\in \mathbb{M}_k(C^*_r(G))$ is the cohomological Laplacian in degree $n$ in  the $\ell_2$-cohomology of $G$. 
We will say that $\Delta_n$ has a spectral gap in a chosen completion of $\mathbb{M}_n(\mathbb{C}G)$ if its spectrum is contained in $\{0\}\cup [\epsilon,\infty)$ for some $\epsilon>0$.
Recall that a group $G$ is of type $F_n$ if it has an Eilenberg-MacLane space 
with a finite $n$-skeleton. Denote by $k_n$ the number of simplices in the chosen model of the Eilenberg-MacLane space.

Let $\beta_{(2)}^n(G)$ denote the $n$-th $\ell_2$-Betti number of $G$.
The following is a consequence of the existence of spectral gap for the Laplacian and \cite{luck}.
\begin{proposition}\label{proposition: Baum-Connes main}
Let $G$ be of type $F_{n+1}$. Assume $\Delta_n \in\mathbb{M}_{k_n}(C^*_r(G))$ has a spectral gap. 
If the Baum-Connes assembly map $K_0^G(\underline{E}G)\to K_0(C^*_r(G))$  
is surjective 
then $$\beta_{(2)}^n(G)\in \Lambda^G.$$  
In particular, if $G$ is torsion-free then $\beta_{(2)}^n(G)\in \mathbb{Z}.$
\end{proposition}

The above theorem illustrates a possible strategy 
for finding counterexamples to the Baum-Connes conjecture: a group $G$ of type $F_{n+1}$ such that 
$\beta_{(2)}^n(G)\notin \Lambda_G$ (in particular, if $\beta_{(2)}^n(G)$ is irrational) with 0 isolated in the spectrum of $\Delta_n$ will not satisfy the 
Baum-Connes conjecture; more precisely, the Baum-Connes assembly map for $G$ will not be surjective.  
We refer to \cite{gomez-aparicio-etal,valette}
for an overview of the Baum-Connes conjecture. It is worth noting that such a counterexample, with bounded orders of finite subgroups, would also not satisfy the Atiyah conjecture, see e.g. \cite{luck-book}.

Our primary application is the construction of classes in the $K$-theory of the 
Roe algebra of a box space of a residually finite group $G$ and their relation to the L\"{u}ck approximation theorem.
Let $\{ N_i\}$ be a family of finite index normal subgroups of a residually finite group $G$, satisfying $\bigcap N_i=\{e\}$. 
Let $\lambda_i$ denote the associated quasi-regular representation of $G$ on $\ell_2(G/N_i)$ and consider the family $\mathcal{N}=\{\lambda, \lambda_1, \lambda_2,\dots\}$ of unitary representations of $G$.
The $C^*$-algebra $C^*_{\mathcal{N}}(G)$ is the completion of the group ring $\mathbb{C}G$ 
with respect to the norm induced by the family $\mathcal{N}$ (see 
Section \ref{subsection: Kazhdan projection related to (T)} for a precise definition). We will denote 
by  $\beta^n(G)=\dim_{\mathbb{C}} H^n(G,\mathbb{C})$ the standard $n$-th Betti number of $G$.

\begin{theorem}\label{theorem: cbcc and l2betti numbers}
Let $G$ be an exact, residually finite group of type $F_{n+1}$ and let $\{ N_i\}$ and $\mathcal{N}$ be as above.
Assume that  $\Delta_n\in  \mathbb{M}_{k_n}(C^*_{\mathcal{N}}(G))$ has a spectral gap
and that the coarse Baum-Connes assembly map 
 $KX_0(Y)\to K_0(C^*(Y))$
for the box space 
$Y=\coprod G/N_i$ of $G$
is surjective. Then
$$ \beta_{(2)}^n(N_i)= \beta^n(N_i)$$ 
for all but finitely many $i$.
\end{theorem}

Note that the conclusion of the above theorem in fact is a strenghtening of L\"{u}ck's approximation theorem \cite{luck-gafa}.
Indeed,  the expression in the formula in the above theorem can be rewritten as 
$$[G:N_i]\left(  \beta_{(2)}^n(G) - \dfrac{\beta^n(N_i)}{[G:N_i]}  \right) = 0,$$
Theorem \ref{theorem: cbcc and l2betti numbers} thus
forces a much stronger equality of the involved Betti numbers. 
On the other hand, as shown in \cite[Theorem 5.1]{luck-approx-survey} there are examples where the  speed of convergence 
of $\beta^n(N_i)/[G:N_i]$ to $\beta_{(2)}^n(G)$ can be as slow as needed. 

Theorem \ref{theorem: cbcc and l2betti numbers} in particular provides new strategies for contradicting the coarse Baum-Connes conjecture. 
It is an important open question wether there exist such counterexamples that do not contain expander graphs. 

 However, it is also natural to conjecture that Theorem \ref{theorem: cbcc and l2betti numbers} could also lead to new counterexamples 
to the Baum-Connes conjecture with coefficients
by embedding a counterexample to the coarse Baum-Connes conjecture, obtained through Theorem \ref{theorem: cbcc and l2betti numbers}, 
isometrically into a finitely generated group, as in \cite{osajda}.
Such constructions could lead to new counterexample to the Baum-Connes conjecture with coefficients, by applying arguments similar to the 
ones in \cite{higson-lafforgue-skandalis}.

We remark that algebraic conditions implying the existence of gaps in the spectrum of the operators arising from the 
cochain complexes with coefficients in unitary representations
were recently provided in
\cite{bader-nowak}. Those conditions involve writing the elements $(\Delta_n^+-\lambda\Delta_n^+)\Delta_n^+$ 
and $(\Delta_n^--\lambda\Delta_n^-)\Delta_n^-$, where $\Delta_n^+=d_n^*d_n$ and $\Delta_n^-=d_{n-1}d_{n-1}^*$ are
 matrices over $\mathbb{R}G$ and the two summands of the Laplacian, as sums of squares in $\mathbb{M}_n(\mathbb{R}G)$. 
 Such conditions can be verified using computational methods.

\subsection*{Acknowledgments}
We are grateful to Dawid Kielak, Roman Sauer, Thomas Schick and Alain Valette for helpful comments. 

This project has received funding from the European Research Council (ERC) under the European Union's Horizon 2020 research and innovation programme (grant agreement no. 677120-INDEX).
\tableofcontents

\section{Higher Kazhdan projections}
 Recall that a group has type $F_n$
if it admits an
Eilenberg-MacLane space $K(G,1)$ with a finite $n$-skeleton. The condition $F_1$ is equivalent to  $G$ being finitely generated, and $F_2$ is equivalent 
to 
$G$ having a finite presentation.

\subsection{Kazhdan projections related to property $(T)$} \label{subsection: Kazhdan projection related to (T)}
We begin with revisiting the classical notion of Kazhdan projections in the context of property $(T)$.
Let $G$ be a finitely generated group with $S=S^{-1}$ a fixed generating set. The real (respectively, complex) group ring of $G$ will be denoted 
by $\mathbb{R}G$ (respectively, $\mathbb{C}G$).
Consider a family $\mathcal{F}$ of unitary representations and let the associated group $C^*$-algebra be the completion
$$C^*_{\mathcal{F}}(G)=\overline{\mathbb{C}G}^{\Vert\cdot\Vert_{\mathcal{F}}},$$
where
$$\Vert f\Vert_{\mathcal{F}}=\sup_{\pi\in\mathcal{F}} \Vert \pi(f)\Vert.$$
We will always assume that $\mathcal{F}$ is a faithful family (i.e. for every $0\neq f\in \mathbb{C}G$ we have $\pi(f)\neq 0$ for some $\pi\in\mathcal{F}$) to ensure $\Vert \cdot \Vert_{\mathcal{F}}$ is a norm.
In particular, we obtain  the maximal group $C^*$-algebra $C^*_{\max}(G)$ if $\mathcal{F}$ is the family of all unitary representations,
and the reduced group $C^*$-algebra $C^*_r(G)$ if $\mathcal{F}=\{ \lambda\}$, where $\lambda$ is the left regular representation.

The following is a classical characterization of property $(T)$ for $G$ due to Akemann and Walter \cite{akemann-walter}. We will use this characterization
as a definition and refer to \cite{bekka-etal} for a comprehensive overview of property $(T)$.
\begin{theorem}[Akemann-Walter \cite{akemann-walter}]
The group $G$ has property $(T)$ if and only if there exists a projection  $p_0=p_0^*=p_0^2\in C^*_{\max}(G)$ 
with the property that for every unitary representation  $\pi$ of $G$
the image $\pi(p_0)$ is the orthogonal projection onto the subspace of invariant vectors of $\pi$.
\end{theorem}
See e.g. \cite[Section 3.7]{higson-roe-book} and \cite{drutu-nowak} for a broader discussion of Kazhdan projections. 
We will now interpret Kazhdan projections in the setting of group cohomology.
Let $\pi$ be a unitary representation of $G$ on a Hilbert space $\HH$.
Denote by $d_0$ the $ \#S \times 1$ matrix 
$$\left[\begin{array}{c}1-s_1 \\\vdots \\1-s_n\end{array}\right],$$
where $s_i$ runs through the elements of $S$, with coefficients in $\mathbb{R}G$
and the Laplacian $\Delta_0\in \mathbb{R}G$ is 
$$\Delta_0 = d_0^*d_0 = 2\left(\#S- \sum_{s\in S} s\right) \ \ \ \ \in \mathbb{R}G.$$
For any unitary representation $\pi$ of $G$ on a Hilbert space $\HH$ we have  $$\ker \pi( \Delta_0) = \ker \pi(d_0) = \HH^{\pi},$$
where $\HH^\pi\subseteq \HH$ denotes the  closed subspace of invariant vectors of $\pi$. 
Note also that the same space can be interpreted as reduced\footnote{Recall that in a setting where the cochain spaces in a cochain complex 
$C^n\ _{\overrightarrow{\ \  d_n\ \ } }C^{n+1}$ are equipped with a Banach space structure the corresponding reduced cohomology is defined as the quotient $\overline{H}^n=\ker d_n\slash \overline{\operatorname{im }d_{n-1}}$, where  $\overline{\operatorname{im} d_{n-1}}$ is the closure of the image of the codifferential $d_{n-1}$. Cohomology is said to be reduced in degree $n$ if $\overline{H}^n=H^n$.}
 cohomology in degree 0:
$$\ker \pi(\Delta_0)= \HH^{\pi} \simeq \overline{H}^0(G,\pi).$$
The image $\pi(p_0)$ of the  Kazhdan projection is the orthogonal projection
$$C^0(G,\pi)\, _{\overrightarrow{\ \ \ \ \ \pi(p_0)\ \ \ \ \ }}\, \ker \pi( \Delta_0).$$
The Kazhdan projection $p_0$ exists in $C^*_{\mathcal{F}}(G)$ if and only if $\pi(\Delta_0)$ has a uniform spectral gap for all $\pi\in \mathcal{F}$.
The projection $p_0$ is non-zero if at least one $\pi\in\mathcal{F}$ has a non-zero invariant vector.

\subsection{Kazhdan projections in higher degree}

We will now discuss a generalization of Kazhdan projections in the setting of higher group cohomology.
Let $G$ be a group of type $F_{n+1}$  with a chosen model $X$ of $K(G,1)$ with finite $n+1$-skeleton. 

As in \cite{bader-nowak} we can consider the cochain complex for cohomology of $G$ with coefficients in $\pi$, where
$$C^n(G,\pi)=\set{f:X^{(n)}\to \HH},$$
 
 By $\mathbb{M}_{m\times m'}(\mathbb{R}G)=\mathbb{M}_{m\times m'}\otimes\mathbb{R}G$ we denote the space of 
 $m\times m'$ matrices with coefficients in the group ring $\mathbb{R}G$.
 When $m=m'$ the resulting algebra is denoted by $\mathbb{M}_{m}(\mathbb{R}G)$.
The codifferentials can be represented by elements  
$$d_n \in \mathbb{M}_{k_{n+1}\times k_{n}}( \mathbb{R}G),$$
where $k_i$ denotes the number of $i$-simplices in $X$,
in the sense that for every unitary representation $\pi$  the codifferential is given by the operator  
$$\pi(d_n):C^n(G,\pi)\to C^{n+1}(G,\pi),$$
where $\pi$ is applied to an element of $\mathbb{M}_{m\times m'}(C^*_{\mathcal{F}}(G))$ entry-wise;
see e.g.  \cite{bader-nowak-deformations,bader-nowak}.

The Laplacian element in degree $n$ is defined to be  
$$\Delta_n=d_n^*d_n+d_{n-1}d_{n-1}^* \in \mathbb{M}_{k_n}(\mathbb{R}G), $$ 
and we denote the two summands by
$$\Delta_n^+=d_n^*d_n,\ \ \ \ \ \Delta_n^-=d_{n-1}d_{n-1}^*,$$
as in \cite{bader-nowak}.
In particular, $\Delta_n$, $\Delta_n^+$, and $\Delta_n^-$ are elements of any completion of $\mathbb{M}_{k_n}( \mathbb{R}G)$ such as   
$\mathbb{M}_{k_n}( C^*_{\max}(G))$ or  $\mathbb{M}_{k_n}(C^*_r(G)).$

For a unitary representation $\pi$ of $G$ we can now consider the operator $\pi(\Delta_n)$. The kernel of $\pi(\Delta_n)$
is the space of harmonic $n$-cochains for $\pi$, and we have the standard Hodge-de Rham isomorphism 
$$\ker \pi(\Delta_n)\simeq \overline{H}^n(G,\pi),$$
between that kernel and the reduced cohomology of $G$ with coefficients in $\pi$ (see e.g. \cite[Section 3]{bader-nowak}).

Similarly, the kernel of the projection $\pi(\Delta_n^+)$ are the $n$-cocycles for $\pi$, and the kernel of $\pi(\Delta_n^-)$ are the 
$n$-cycles.

\begin{definition}[Higher Kazhdan projections]
A Kazhdan projection in degree $n$ is a projection $p_n=p_n^*=p_n^2\in \mathbb{M}_{k_n}(C^*_{\mathcal{F}}(G))$ 
such that
for every  unitary representation $\pi\in \mathcal{F}$ the projection $\pi(p_n)$ is the orthogonal projection $C^n(G,\pi)\to \ker \pi(\Delta_n)$.

A partial Kazhdan projection in degree $n$ is a projection $p_n^+, p_n^- \in \mathbb{M}_{k_n}(C^*_{\mathcal{F}}(G))$, such 
that for every unitary representation $\pi\in\mathcal{F}$ the projection $\pi(p_n^+)$ (respectively, $\pi(p_n^-)$) is the 
orthogonal projection onto the kernel of $\pi(d_n)$ (respectively, onto the kernel of $\pi(d_{n-1}^*)$). 
\end{definition}

In degree 0 and when $\mathcal{F}$ is the family of all unitary representations the projection $p_0\in C^*_{\max}(G)$ is the classical Kazhdan projection.
Clearly, $p_n$ is non-zero if and only if for at least one $\pi\in\mathcal{F}$ we have $\pi(p_n)\neq 0$.

From now on we will shorten the subscript $k_n$, denoting the number of $n$-simplices in the chosen $K(G,1)$, to $k$, as the dimension will
be clear from the context.
Similarly as in the case of property $(T)$, higher Kazhdan projections exist in the presence of a spectral gap.

\begin{proposition}
Assume that $\Delta_n \in \mathbb{M}_k (C^*_{\mathcal{F}}(G))$ (respectively,  $\Delta_n^+$, $\Delta_n^-$) has a spectral gap.
Then the Kazdhan projection $p_n$ (respectively, the partial Kazhdan projections $p_n^+$, $p_n^-$) exist in the $C^*$-algebra 
$\mathbb{M}_{k_n}(C^*_{\mathcal{F}}(G))$.
\end{proposition}

\begin{proof}
By assumption, the spectrum of $\Delta_n$ satisfies 
\begin{equation*}\label{equation: spectral gap}
\sigma(\pi(\Delta_n))\subseteq \{0\}\cup [\epsilon,\infty)
\end{equation*}
for every unitary representation $\pi\in \mathcal{F}$ and some $\epsilon>0$. 
In particular, $\sigma(\Delta_n)\subseteq \{0\}\cup [\epsilon,\infty)$
when $\Delta_n$ is viewed as an element of $\mathbb{M}_k (C^*_{\mathcal{F}}(G))$.

Due to the presence of the spectral gap we can apply continuous functional calculus and conclude that the limit of the heat semigroup
$$\lim_{t\to \infty} e^{-t\Delta_n}$$ is the spectral projection  $p_n \  \in \mathbb{M}_k(C^*_{\mathcal{F}}(G))$.
By construction, $p_n$ has the required properties.

The proof for the partial projections is analogous. 
\end{proof}
Note that the fact that the spectrum of $\Delta_n \in \mathbb{M}_k(C^*_{\mathcal{F}}(G))$ has a gap implies that both $\Delta_n^+$ and $\Delta_n^-$ 
also have spectral gaps and in particular all three higher Kazhdan projections exist in 
$\mathbb{M}_{k}(C^*_{\mathcal{F}}(G))$.
We have the following decomposition for higher Kazhdan projections.

\begin{lemma}
Assume that   $\Delta_n \in \mathbb{M}_k(C^*_{\mathcal{F}}(G))$ has a spectral gap. Then 
$$p_n=p_n^+p_n^-.$$
\end{lemma}
\begin{proof}
It suffices to notice that $\Delta_n^+$ and $\Delta_n^-$ commute in  $\mathbb{M}_k(C^*_{\mathcal{F}}(G))$ and satisfy
$$\Delta_n^+\Delta_n^-=\Delta_n^-\Delta_n^+=0,$$
using the property that $d_id_{i-1}=0$ for every $i\in \mathbb{N}$. Consequently, the corresponding spectral projections commute
and their product is the projection onto the intersections of their kernels, which is precisely the kernel of $\Delta_n$.
\end{proof}

\subsection{$K$-theory classes - a special case}

In this section we consider the case when $\mathcal{F}=\{ \lambda\}$, where $\lambda$ is the left regular 
representation of $G$ on $\ell_2(G)$ and $C^*_{\mathcal{F}}(G)=C^*_r(G)$ is the reduced group $C^*$-algebra 
of $G$. We refer to \cite{higson-roe-book, valette} for background on $K$-theory and assembly maps.

The projection $p_n\in \mathbb{M}_k(C^*_r(G))$ exists if $\lambda(\Delta_n)$ has a spectral gap at $0$.
In this case the kernel of the Laplacian $\lambda(\Delta_n)$ is isomorphic to the reduced $\ell_2$-cohomology of $G$:
$$\ker \lambda(\Delta_n)\simeq \ell_2\overline{H}^n(G) \simeq \ker \lambda(d_n)\big\slash \,\overline{ \operatorname{im} \lambda(d_{n-1})},$$
and it is non-trivial if and only the $n$-th $\ell_2$-Betti number $\beta_{(2)}^n(G)$ is non-zero.

The projection $p_n\in \mathbb{M}_k( C^*_r(G))$ gives rise to a $K$-theory class $[p_n]\in K_0(C^*_r(G))$.
We are interested in the properties of the classes defined by the higher Kazhdan projections $[p_n]\in K_0(C^*_r(G))$.

Recall that the canonical trace $\tau_G$ on $C^*_r(G)$ is given by 
$$\tau_G(\alpha)=\langle \alpha \delta_e,\delta_e\rangle,$$
where $\delta_e\in \ell_2(G)$ is the Dirac delta at the identity element $e\in G$. For an element $\alpha\in \mathbb{C}G$ we have $\tau(\alpha)=\alpha(e)$.
The induced trace on $\mathbb{M}_k (C^*_r(G))$ is defined as 
$$\tau_{k,G}\left((a_{ij})\right) = \sum_{i=1}^k \tau_G(a_{ii}),$$
and there is an induced map 
$$\tau_*: K_0(C^*_r(G))\to \mathbb{R}$$ 
defined by $\tau_*([p])=\tau_{k,G}(p)$,
for a class represented by a projection $p\in \mathbb{M}_k(C^*_r(G))$.

\begin{proposition}\label{proposition: existence of higher Kazhdan projection in M_n(C* reduced)}
Assume that  $\Delta_n$ has a spectral gap  in $\mathbb{M}_k(C^*_r(G))$. 
Then $$\tau_*([p_n])= \beta_{(2)}^n(G).$$
In particular, if $\beta_{(2)}^n(G)\neq 0$ then 
 $[p_n] \neq 0$ in  $K_0(C^*_r(G))$.
\end{proposition}

\begin{proof}
By definition, the $n$-th $\ell_2$ Betti number of $G$ is the von Neumann dimension (over the group von Neumann algebra $L(G)$) 
of the 
orthogonal projection onto the kernel of the Laplacian. Under our assumptions the projection 
is now an element of the smaller algebra $\mathbb{M}_k(C^*_r(G))\subseteq \mathbb{M}_k(L(G))$ and
$$\beta_{(2)}^n(G) = \tau_{k,G}(p_n).$$
Therefore, the induced trace on the group $K_0(C^*_r(G))$ satisfies $\tau_*([p_n])= \beta_{(2)}^n(G)$,
as claimed
\end{proof}

A similar argument shows non-vanishing of the classes represented by the partial projections.
\begin{proposition}
Assume that $0$ is isolated in the spectrum of $\Delta_n^+$ (respectively, $\Delta_n^-$) in $\mathbb{M}_k(C^*_r(G))$. 
If $\beta_{(2)}^n(G)\neq 0$ then
 $[p_n^+] \neq 0$ (respectively, $[p_n^-]\neq 0$) in  $K_0(C^*_r(G))$.
\end{proposition}
\begin{proof} 
By assumption $p_n^+$ exists in $\mathbb{M}_k(C^*_r(G))$ and $p_n$ exists in $\mathbb{M}_k(L(G))$ and we will compare the 
values of the trace over the latter algebra.
Since $\tau_{k,G}(p_n)\neq 0$ and 
$$\operatorname{im}(p_n)\subseteq \operatorname{im} (p_n^+)$$
by monotonicity of the von Neumann dimension over $L(G)$
we have 
$$\tau_{k,G}(p_n^+)\ge \tau_{k,G}(p_n)>0.$$
Consequently $\tau_*([p_n^+])\neq 0$ in $K_0(C^*_r(G))$.
\end{proof}

\subsection{The Baum-Connes assembly map}

The Baum-Connes assembly map is the map 
$$\mu_i:K_i^{G}(\underline{E}G)\to K_i(C^*_r(G)),$$
$i=0,1$.
The trace conjecture of Baum and Connes and the modified trace conjecture by L\"{u}ck \cite{luck} predict the range of the composition
$\tau_*\circ\mu_0$.
In particular, the modified trace conjecture, formulated and proved by L\"{u}ck \cite{luck},
states that  $\tau_*\circ\mu_i$ takes values in a subring $\Lambda^G\subseteq \mathbb{Q}$, 
generated from $\mathbb{Z}$ by adjoining the inverses of cardinalities of finite subgroups.

\begin{theorem}[L\"{u}ck {\cite[Theorem 0.3]{luck}}]
If the Baum-Connes assembly map is surjective then the composition $\tau_*\circ \mu_0$ takes values in the ring 
$\Lambda^G\subseteq \mathbb{Q}$.
\end{theorem}

We can now state the relation between the Baum-Connes conjecture and $\ell_2$-Betti numbers.

\newcounter{aaa}
\setcounter{aaa}{\value{theorem}}

\setcounter{theorem}{0}
\begin{proposition}
Let $G$ be of type $F_{n+1}$. Assume that $0$ is isolated in the spectrum of $\Delta_n\in \mathbb{M}_k(C^*_r(G)) $. 
If the Baum-Connes assembly map $K_0^G(\underline{E}G)\to K_0(C^*_r(G))$  
is surjective 
then $$\beta_{(2)}^n(G)\in \Lambda^G.$$  
In particular, if $G$ is torsion-free then $\beta_{(2)}^n(G)\in \mathbb{Z}.$
\end{proposition}
\setcounter{theorem}{\value{aaa}}

\begin{remark}[The Euler class]\normalfont \label{remark: Euler class}
Consider an infinite group $G$ such that $K(G,1)$ can be chosen to be a finite complex.
Assume that $\lambda(\Delta_n)$ has a spectral gap in $\mathbb{M}_k(C^*_r(G))$ for every $n\ge 0$ 
(in particular, $G$ is non-amenable). 
Then we can define the Euler class 
$$\Xi(G)=\sum_{i\ge 0}(-1)^i [p_i]$$
in the $K$-theory $K_0(C^*_r(G))$.
We have that 
$$\tau_*(\Xi(G))=\chi_{(2)}(G)$$ 
is the $\ell_2$-Euler characteristic of the chosen $K(G,1)$.

Atiyah's $L_2$-index theorem associates the $L_2$-index, analogous to the Fredholm index but defined via the trace on the von Neumann algebra,
to a lift $\widetilde{D}$ of an elliptic operator $D$ on a compact manifold $M$ to the universal cover of $M$.
Loosely speaking, in the presence of the spectral gap the $L_2$-indices of certain operators appearing in the  Atiyah $L_2$-index theorem 
are in fact traces of their Baum-Connes indices. 
\end{remark}

\subsection{Examples}
In case of an infinite group $G$ the projection $\lambda(p_0)\in C^*_r(G)$ is always 0 since the kernel of $\lambda(\Delta_0)$ 
consists precisely of constant functions in $\ell_2(G)$.

We will now discuss certain cases in which higher Kazhdan projections exist.
We will use the following fact, which is a reformulation of \cite[Proposition 16, 2) and 3)]{bader-nowak}
\begin{lemma}\label{lemma : laplacian spectral gap if cohomologies reduced}
Let $\pi$ be a unitary representation of $G$.
The Laplacian $\pi(\Delta_n)$ has a spectral gap around $0$ if and only if the cohomology groups $H^n(G,\pi)$ and $H^{n+1}(G,\pi)$ are both reduced.
\end{lemma}

\begin{proof}[Sketch of proof]
Using the notation of \cite{bader-nowak}, given a cochain complex 
$$\dots \to C^{n-1}\to C^n\to C^{n+1}\to \dots,$$
where the $C^n$ have Hilbert space structures and the codifferentials $d_n:C^n\to C^{n+1}$ are bounded operators, 
we have $\Delta_n=d_n^*d_n+d_{n-1}d_{n-1}^*$.
We have an orthogonal decomposition
$$C^n=C_n^+\oplus \ker \Delta_n \oplus C_n^-,$$
where $C_n^-=\ker d_n \cap \ker (d_{n-1}^*)^\perp$ and $C_n^+=\ker d_{n-1}^*\cap( \ker d_n^*)^\perp$.
Then the restritictions of $d_n^*d_n$ and of $d_{n-1}d_{n-1}^*$ are invertible on $C_n^+$ and $C_n^-$, respectively if and only if 
the cohomology groups $H^n$ and $H^{n+1}$ are reduced \cite{bader-nowak}. Clearly, the first condition is equivalent to the fact that $0$ is isolated in 
the spectrum of the Laplacian $\Delta_n$.
\end{proof}

\subsubsection{Free groups}
Consider the free group $F_n$ on $n\ge 2$ generators. The standard Eilenberg-MacLane space $K(F_n,1)$ is the wedge $\bigvee_n S^1$ of $n$ circles,
whose universal cover is the tree, the Cayley graph of $F_n$. Then for $F_n$ the projection $p_1\in \mathbb{M}_k(C^*_r(F_n))$ exists and gives rise to a non-zero class in 
$K$-theory $C^*_r(F_n)$.

Indeed, since for $n\ge 2$, the free group $F_n$ is non-amenable we have that the $\ell_2$-cohomology
is reduced in degree 1 (i.e., the range of the codifferential into the 1-cochains is closed). 
Since $K(F_n,1)$ is 1-dimensional, the cohomology in degree 2 is reduced as there are no 2-cochains. 
We thus have $d_1=0$, $p_1^+=\operatorname{I}$ and $p_1=p_1^-$.
These facts together with lemma 
\ref{lemma : laplacian spectral gap if cohomologies reduced} imply that the cohomological Laplacian $\Delta_1$ in degree 1 has a spectral gap.

The $\ell_2$-Betti number of a free group $F_n$ on $n$ generators is
$$\beta_{(2)}^1(F_n)=n-1.$$ 
By the previous discussion, 
$$[p_1]\neq 0 \ \ \ \text{ in }\ \ \  K_0(C^*_r(F_n)).$$

Moreover, we can identify precisely the class of $p_1$ in the $K$-theory group.
Indeed, $K_0(C^*_r(F_n))$ is isomorphic to $\mathbb{Z}$ and generated by $1\in C^*_r(F_n)$. Because of this the value of the trace
determines the class of $p_1$ to be 
$$[p_1]=n-1\ \ \ \ \in \ \ \ \ \mathbb{Z}\simeq K_0(C^*_r(F_n)).$$
The corresponding Euler class  (see Remark \ref{remark: Euler class})  in $K_0(C^*_r(F_n))$ is given by
$$\Xi(F_n)=[p_0] - [p_1]= -[p_1].$$
More generally, the argument used in the above example together with  \cite[Corollary 4.8]{kahlerbook}  gives 
the following.
\begin{corollary} 
Let $G$ be a finitely presented group with infinitely many ends such that $H^2(G,\ell_2(G))$ is reduced. Then the projection 
$p_1(G)$ exists in $\mathbb{M}_{k_n}(C^*_r(G))$ and 
$$[p_1]\neq 0 \ \ \ \text{ in }\ \ \  K_0(C^*_r(G)).$$
\end{corollary}

\subsubsection{K\"{a}hler groups}
As shown in \cite{efremov}, \cite[Remark 2]{gromov-shubin}, see also \cite[Theorem 2.68]{luck-book} or \cite[Section 8]{luck-survey},
the property that the Laplacian has a spectral gap in the analytic setting (i.e. $L_2$-cohomology defined in terms differential forms) 
 is equivalent to the existence of the spectral gap in the combinatorial setting.
This is most often phrased in terms of Novikov-Shubin invariants associated to spectral density functions: spectral gap is equivalent to 
the Novikov-Shubin invariant being infinite, which is the same as the property that the spectral density function  is constant in the neighborhood 
of 0. Then it is shown that the combinatorial and analytic spectral density functions are dilation equivalent, and in particular, equal in a neighborhood of 0.

Consider now a closed K\"{a}hler hyperbolic manifold $M$  and denote $G=\pi_1(M)$.
As shown by Gromov \cite{gromov-Kahler}, see also \cite[Section 11.2.3]{luck-book}, the Laplacian acting on $L_2$-differential forms on the universal cover 
$\widetilde{M}$ has a spectral gap in every dimension. Gromov also showed that the $\ell_2$-cohomology of $M$ 
vanishes except for the middle dimension.

Consequently, if $G$ is a fundamental group of such a K\"{a}hler manifold then the assumptions of 
Proposition \ref{proposition: existence of higher Kazhdan projection in M_n(C* reduced)}
 are satisfied and 
the projection $p_n$ exists in $\mathbb{M}_{k_n}(C^*_r(G))$ in every dimension. 

Moreover, the Euler class 
$$\Xi(G)=\sum_{i\ge 0} (-1)^i[p_i] \ \ \ \ \in K_0(C^*_r(G))$$
exist and its trace is the $\ell_2$-Euler characteristic of $G$.

\subsubsection{Lattices in $\operatorname{PGL}_{n+1}(\mathbb{Q}_p)$}
As shown in \cite[Proposition 19]{bader-nowak}, given $n\ge 2$, a sufficiently large prime $p$ and a lattice $\Gamma \subseteq \operatorname{PGL}_{n+1}(\mathbb{Q}_p)$ we have that 
$H^n(\Gamma,\pi)$ is reduced for every unitary representation $\pi$. This fact is equivalent to $\pi(\Delta_n^-)$ having a spectral gap for every unitary
representation $\pi$ of $\Gamma$ and it is easy to check that such a spectral gap must be then uniform; i.e., $\Delta_n^+$ has a spectral
gap in $\mathbb{M}_k(C^*_{\max}(\Gamma))$.

At the same time there exists a finite index subgroup $\Gamma'\subseteq \Gamma$
such that $H^n(\Gamma',\mathbb{C})\simeq H^n(\Gamma,\ell_2(\Gamma/\Gamma'))\neq 0$. 
As a consequence we obtain the following.
\begin{proposition}
Let $\Gamma\subseteq \operatorname{PGL}_{n+1}(\mathbb{Q}_p)$ be a lattice. Then the partial Kazhdan projection $p_n^+$
exists in $\mathbb{M}_k(C^*_{\max}(\Gamma))$ and is non-zero.
\end{proposition}

\section{The coarse Baum-Connes conjecture}
We will now prove our main results and  describe the $K$-theory classes induced by the higher Kazhdan projections in Roe algebras and their 
applications to the coarse Baum-Connes conjecture of a box space 
of a residually finite group. The arguments we will use were introduced in \cite{higson,higson-lafforgue-skandalis} and developed further in 
\cite{willett-yu1}. We also refer to \cite[Chapter 13]{willett-yu-book} for more details.

\subsection{Roe algebras, box spaces and Laplacians}
\subsubsection{Roe algebras} 
The Roe algebra $C^*(X)$ of a discrete, bounded geometry metric space $X$ is the completion in $B(\ell_2(X;\mathcal{H}_0))$ 
of the $*$-algebra $\mathbb{C}[X]$ 
of finite propagation operators, which are locally compact; i.e. for a discrete space the
matrix coefficients are compact operators $T_{x,y}\in \mathcal{K}(\mathcal{H}_0)$ on a fixed, infinite-dimensional Hilbert space $\mathcal{H}_0$.

If $G$ acts on $X$ then the equivariant Roe algebra is defined to be the closure of the subalgebra 
$\mathbb{C}[X]^G\subseteq \mathbb{C}[X]$ of equivariant finite propagation operators, i.e. satisfying 
$T_{x,y}=T_{gx,gy}$ for any $g\in G$ and $x,y\in X$.

The uniform Roe algebra $C^*_u(X)$ of $X$ is the completion in $B(\ell_2(X))$ of the $*$-algebra $\mathbb{C}_u[X]$ of 
finite propagation operators. The equivariant uniform Roe algebra $\mathbb{C}_u^*(X)^G$ is defined as before by considering $G$-invariant operators in
$\mathbb{C}_u[X]$ and taking the closure inside $C^*_u(X)$.

See Definitions 3.2 and 3.6 in \cite{willett-yu1} for a detailed  description of the Roe algebra and the equivariant Roe algebra.

\subsubsection{Laplacians on box spaces}
Consider a finitely generated residually finite group $G$. 
Let $\{N_i\}$ be a family of finite index normal subgroups of $G$ satisfying $\bigcap N_i=\{e\}$. Consider the space 
$$Y=\coprod G/N_i,$$
viewed as a box space with $d(G/N_i,G/N_j)\ge 2^{i+j}$.

Let $\lambda_i$ be the quasi-regular representation of $G$ on $\ell_2\left(G/N_i\right)$ given by pulling back the regular representation
of $G/N_i$ on $\ell_2(G/N_i)$ to $G$ via the quotient map $G\to G/N_i$ and denote by $\mathcal{N}$ the family 
of representations
$$\mathcal{N}=\{\lambda, \lambda_1,\lambda_2,\dots\}.$$
The standing assumption in this section will be that the cohomological $n$-Laplacian $\Delta_n$ has a spectral gap in 
$\mathbb{M}_k(C^*_{\mathcal{N}}(G))$.
The cohomological Laplacian element  $\Delta_n\in \mathbb{M}_k(\mathbb{C}G)$ in degree $n$ maps to $\lambda_i(\Delta_n)$ in
$\mathbb{M}_k(\lambda_i(C^*_{\mathcal{N}}(G)))$ and we will denote 
$$\Delta_n^i= \lambda_i(\Delta_n).$$
Define  the Laplace element of the box space  to be 
$$D_n=\oplus_i\ \Delta_n^i\in \bigoplus \mathbb{M}_k ( \mathbb{C}_u[G/N_i])
\subseteq \mathbb{M}_n ( \mathbb{C}_u[Y]).$$
Note that $D_n$ has a spectral gap in $C^*(Y)$ if and only if all the $\Delta_n^i$  have a spectral gap in $C^*_{\mathcal{N}}(G)$.

\subsection{The lifting map}
From now on we will assume the group $G$ satisfies the operator norm localization property. 
As shown by Sako \cite{sako} the operator norm localization property is equivalent to exactness for bounded geometry 
metric spaces, in particular for finitely generated groups, so the assumptions of Theorem \ref{theorem: cbcc and l2betti numbers} guarantee this property
for $G$.

Note that we will only be using a special case of the setup described in \cite{willett-yu1}, as our 
box space consists of a family of finite quotients of a single group $G$, which naturally serves as a covering space for all of the finite quotients.
This setting is the same as the one in \cite{higson}, however we do follow the detailed version of the arguments as in \cite{willett-yu1}.

Willett and Yu  in \cite[Lemma 3.8]{willett-yu1} define a lifting map, a $*$-homomorphism
$$\phi: \mathbb{C}[Y]\to \dfrac{\prod_i \mathbb{C} [G]^{N_i}}{\bigoplus_i \mathbb{C} [G]^{N_i}},$$
into the algebra of sequences of elements of the $N_i$-invariant subspaces of $\mathbb{C}[G]$ modulo the relation of 
being equal at all but finitely many entries. 

We have 
$$\mathbb{M}_k\left(\dfrac{\prod_i \mathbb{C} [G]^{N_i}}{\bigoplus_i \mathbb{C} [G]^{N_i}}\right) =
\dfrac{\prod_i \mathbb{M}_k\left( \mathbb{C} [G]^{N_i}\right)}{\bigoplus_i  \mathbb{M}_k\left(\mathbb{C} [G]^{N_i}\right)}$$
and for every $k\in\mathbb{N}$ the map $\phi$ induces a $*$-homomorphism 
$$\phi^{(k)}: \mathbb{M}_k(\mathbb{C}[Y]) \to \mathbb{M}_k\left(\dfrac{\prod \mathbb{C} [G]^{N_i}}{\bigoplus \mathbb{C} [G]^{N_i}}\right),$$
by applying $\phi$ entrywise. 

We recall  that in \cite[Lemma 3.12]{willett-yu1} it was shown that if $G$ has the operator norm localization property then the map 
$\phi$ extends to a map 
$$\phi: C^*(Y)\to \dfrac{\prod_i C^*(G )^{N_i} }{\bigoplus_i C^*(G)^{N_i} },$$
and similarly induces the corresponding map $\phi^{(n)}$ on the $k\times k$ matrices over these algebras.
Here,  $\dfrac{\prod_i C^*(G )^{N_i} }{\bigoplus_i C^*(G )^{N_i} }$ is the algebra of sequences of elements of $C^*(G)^{N_i}$ modulo the relation of 
being asymptotically equal.

The same formula as the one for $\phi$ in \cite[Lemma 3.8]{willett-yu1} gives a uniform version of the lifting map, 
$$\phi_u: \mathbb{C}_u[Y]\to \dfrac{\prod_i \mathbb{C}_u [G]^{N_i}}{\bigoplus_i \mathbb{C}_u [G]^{N_i}},$$
which is in the same way a $*$-homomorphism.

If $G$ has the operator norm localization property (see e.g. \cite{chen-etal, sako}) then it also has the uniform version of that property, as proved by Sako \cite{sako}.
This implies that $\phi_u$ extends to a $*$-homomorphism
$$\phi_u: C^*_u(Y)\to \dfrac{\prod_i C^*_u (G)^{N_i}}{\bigoplus_i C^*_u (G)^{N_i}}.$$
See 13.3.11 and 13.3.12 in \cite{willett-yu-book} for the detailed arguments, which apply verbatim here.
Similarly, $\phi_u^{(n)}$ in both cases induces a map on matrices over the respective algebras.

\subsection{Surjectivity of the coarse Baum-Connes assembly map}
We will now focus on the lift of the Laplacian $D_n$.
In order to do this we will use an argument based on Lemma 5.6 in \cite{willett-yu1}. 
\begin{lemma}
Let $G$ be an exact group and assume that $\Delta_n$ has a spectral gap in $C^*_{\mathcal{N}}(G)$.
Denote by 
$$P_n=\bigoplus \lambda_i(p_n) \ \ \ \ \in \mathbb{M}_k( C_u^*(Y))$$ 
the spectral projection associated to $D_n\in \mathbb{M}_k(C_u^*(Y))$.
Then 
$$\phi_u^{(k)}(P_n)=\prod_{i=1}^{\infty}\lambda(p_n),$$ 
where $\lambda(p_n)$ is the projection onto the harmonic $n$-cochains in the $\ell_2$-cohomology of $G$.
\end{lemma}
\begin{proof}
The projection $P_n$ is the spectral projection of $D_n$.
The Laplacian element $D_n$ defined above lifts to an element 
$$\phi_u^{(k)}(D_n)\in \mathbb{M}_k\left(\dfrac{\prod_{i=1}^{\infty} \mathbb{C}_u [G]^{N_i}}{\bigoplus_{i=1}^{\infty} \mathbb{C}_u [G]^{N_i}}\right),$$
whose spectrum is contained in $\{ 0\} \cup [\epsilon, \infty)$ for some $\varepsilon>0$. Indeed, since the same 
is true for $D_n \in C^*_u(Y)$, by assumption, and spectral gaps are preserved by homomorphisms of unital $C^*$-algebras.

It is easy to see that from the definition of $\phi$ in the proof of \cite[Lemma 3.8]{willett-yu1} that the lift of $D_n$
is represented by the constant sequence 
$$\phi_u^{(k)}(D_n)=\prod_{i=1}^{\infty}\Delta_n,$$
in the algebra
$$\dfrac{\prod_{i=1}^{\infty} \mathbb{M}_k(\mathbb{C}G )}{\bigoplus_{i=1}^{\infty} \mathbb{M}_k(\mathbb{C}G)} \subseteq \mathbb{M}_k\left(\dfrac{\prod_{i=1}^{\infty} \mathbb{C}_u [G]^{N_i}}{\bigoplus_{i=1}^{\infty} \mathbb{C}_u [G]^{N_i}}\right). $$
Indeed, this follows from the fact that the size of the support of $\Delta_n^i$ and the formula for the lift $\phi_{(u)}$ are both independent 
of the particular quotient $G/N_i$ for $i$ sufficiently large.

Consider now the spectral projection $\widetilde{P}_n$ associated to $\phi_u^{(k)}(D_n)$.  Then
$$\widetilde{P}_n=\lim_{t\to \infty} e^{-t\phi_u^{(k)}(D_n)}=\lim_{t\to \infty} \phi_u^{(k)}\left(e^{-tD_n}\right)=\phi_u^{(k)}(P_n).$$
Thus the associated spectral projection is of the form  $\prod_{i=1}^{\infty} \lambda(p_n)$, where 
 $$\lambda(p_n)\in \mathbb{M}_k(C^*_r(G))$$ 
 is the 
 projection onto the harmonic $n$-cochains in the $\ell_2$-cochains of $G$, as claimed.
\end{proof}

We now define the \emph{higher Kazhdan projection of the box space $Y$} to be $P_n\otimes q\in \mathbb{M}_n( C^*(Y))$, where $q$ is any rank one projection 
on $\mathcal{H}_0$. We consider the  associated $K$-theory class $[P_n\otimes q]\in K_0(C^*(Y))$. We will show that under certain conditions
this class is non-zero and does not lie in the image of the coarse Baum-Connes conjecture for $Y$.

The first thing to notice is that the lift of the projection $P_n\otimes q$ can be described explicitly as an element of $C^*(G)^G$.

\begin{lemma}
Let $q\in \mathcal{K}(\mathcal{H}_0)$. Then 
$$\phi(P_n\otimes q) = \phi_u(P_n)\otimes q.$$
\end{lemma}
\begin{proof}
Let $\alpha_i$ be a sequence of finite propagation operators such that 
$P_n=\lim_i \alpha_i$.
Then 
$$\phi(\alpha_i\otimes q) = \phi_u(\alpha_i)\otimes q,$$
by the definition  of $\phi$ (\cite[Lemma 3.8]{willett-yu1}).
Passing to the limit and using continuity of $\phi$ we obtain the claim.
\end{proof}

Following \cite{higson, willett-yu1} define the map $d_*: K_0(C^*(Y))\to \dfrac{\prod \mathbb{Z}}{\bigoplus \mathbb{Z}}$ by
taking the map 
$$d: C^*(Y)\to \dfrac{\prod \mathcal{K}(\ell_2(G/N_i;\mathcal{H}_0))}{\bigoplus \mathcal{K}(\ell_2(G/N_i;\mathcal{H}_0))},$$
defined by 
$$A\mapsto \prod_{i=1}^{\infty} Q_i A Q_i,$$
where $Q_i:\ell_2(Y)\to \ell_2(G/N_i)$ is the projection on the $i$-th box $G/N_i$,
and considering the induced map on $K$-theory
$$d_*:K_0(C^*(Y))\to K_0\left(\dfrac{\prod \mathcal{K}(\ell_2(G/N_i;\mathcal{H}_0))}{\bigoplus \mathcal{K}(\ell_2(G/N_i;\mathcal{H}_0))}\right) \simeq 
\dfrac{\prod \mathbb{Z} }{\bigoplus \mathbb{Z}}.$$
The above map $d_*$ can recognize when our $K$-theory class is non-zero. The map $d_*$ applied to 
the projection $P_n$ gives the sequence of dimensions of images of $\lambda_i(p_n)$.

On the other hand, consider the trace $\tau_i$ on $C^*(G)^{N_i}\simeq C^*_r(N_i)\otimes \mathcal{K}$
by considering the tensor product of the standard trace on $C^*_r(N_i)$ with the canonical trace on $\mathcal{K}$.
These $\tau_i$ induce the map 
$$T:K_0\left( \dfrac{\prod_i C^*(G)^{N_i} }{\bigoplus_i C^*(G)^{N_i} } \right)\to \dfrac{\prod \mathbb{R}}{\bigoplus \mathbb{R}}.$$
The following diagram was analyzed in \cite{willett-yu1}.
\begin{center}
\begin{tikzcd}
  & K_0\left( \dfrac{\prod_i C^*(G)^{N_i} }{\bigoplus_i C^*(G)^{N_i} } \right)\arrow[rd, "T"]&  \\
 KX_0(Y) \arrow[ru, "\widetilde{\mu}_c"] \arrow[rd, "\mu_c"] &&\dfrac{\prod \mathbb{R}}{\bigoplus \mathbb{R}}\\
  &K_0(C^*(Y)) \arrow[uu, "\phi_*" ] \arrow[ru, "d_*"]
\end{tikzcd}
\end{center}

The next lemma was formulated in \cite{willett-yu1} for projections in $C^*(X)$, however the proof applies equally to projections in 
$\mathbb{M}_n(C^*(X))$.
\begin{lemma}[Lemma 6.5 in \cite{willett-yu1}]
If $p$ is a projection in $\mathbb{M}_n(C^*(Y))$ such that the class $[p]\in K_0(C^*(Y))$ is in the image of the coarse assembly map 
$\mu_0:KX_0(Y)\to K_0(C^*(Y))$, then
$$d_*([p])=  T(\phi_*([p]))\ \ \ \ \ \ \text{ in } \dfrac{\prod \mathbb{R}}{\bigoplus \mathbb{R}}.$$
\end{lemma}
These maps take values in  $\prod \mathbb{R}/\oplus \mathbb{R}$ viewed as an object in the category of abelian groups; that is, the equality of the two traces is an equality  
of coordinates for all but finitely many $i$ (see Section 13.3 in \cite{willett-yu-book}).
The next statement allows to show that when it is non-compact, the projection $P_n$ gives rise to a non-zero $K$-theory class in 
$K_0(C^*(Y))$.

\begin{proposition}\label{proposition: lower trace}
Assume that $H^n(G, \ell_2(G/N_i))\neq 0$ for infinitely many $i\in\mathbb{N}$. 
Then $d_*[P_n]\neq 0$.
\end{proposition}

\begin{proof}We have
\begin{align*}
d_*(P_n)&= \prod_{i=1}^{\infty}  \dim \lambda_i(p_n)= \prod_{i=1}^{\infty} \dim H^n(G; \lambda_i).
\end{align*}
Indeed, since 
$$\mathbb{M}_n(C^*(Y)) = C^*(\mathbb{Z}_n\times Y),$$
the map
$$d^{(n)}:\mathbb{M}_n(C^*(X)) \to  
\dfrac{\prod \mathbb{M}_n(\mathcal{K}(\ell_2(G/N_i;\mathcal{H}_0)))}{\bigoplus\mathbb{M}_n( \mathcal{K}(\ell_2(G/N_i;\mathcal{H}_0)))}.$$
can be rewritten as 
$$d:C^*(X\times \mathbb{Z}_n) \to \dfrac{\prod \mathcal{K}(\ell_2(G/N_i\times \mathbb{Z}_n;\mathcal{H}_0))}{\bigoplus \mathcal{K}(\ell_2(G/N_i\times \mathbb{Z}_n;\mathcal{H}_0))}.$$
By the assumption on cohomology of $G$ with coefficients in $\lambda_i$, the projection $p$ is non-compact in $C^*(Y\times \mathbb{Z}_n)$
and as shown in the proof of Theorem 6.1 on page 1407 in \cite{willett-yu1}, we have $d_*[p]\neq 0$.
\end{proof}

The next lemma shows that the trace of the lift is naturally related to $\ell_2$-Betti numbers.
\begin{lemma}\label{lemma : computation of the upper trace of P_n}
$T(\phi_*([P_n]))= \prod_{i=1}^{\infty}[G:N_i] \beta_{(2)}^n(G)=\prod_{i=1}^{\infty}\, \beta_{(2)}^n(N_i)$.
\end{lemma}
\begin{proof}
For a finite index subgroup $N\subseteq G$ the trace $\tau_N$ on $C^*(G)^N$ is defined as the tensor product of the canonical trace $\operatorname{tr}_N$ on $C^*_r(N)$
with the the canonical (unbounded trace) $\operatorname{Tr}$ on the compact operators, via the isomorphism
$$\psi_N:{C^*(G)^N\ }_{\overrightarrow{\ \ \ \simeq\ \ \ \ }} \ C^*_r(N)\otimes \mathcal{K}.$$
The isomorphism $\psi_N$ is defined by considering a fundamental domain $D\subset G$ for $N\subseteq G$
and for $A\in C^*(G)^N$ identifying 
$$A\mapsto \sum_{g\in N} u_g \otimes A^{(g)},$$
where $A^{(g)}\in \mathcal{K}(\ell_2(D,\mathcal{H}_0))$ is defined by the formula
$$A^{(g)}_{x,y} = A_{x,gy},$$
for $x,y \in D$.

With this identification we observe that for an element $A \in C^*(G)^G$ of the form $A =\alpha\otimes q$, where $\alpha\in C^*_r(G)$
and $q$ is a rank 1 projection on $\mathcal{H}_0$, we have
\begin{align*}
\tau_N(A) &= \tau_N\left(\sum_{g\in N} u_g \otimes A^{(g)}\right) \\
&= \sum_{g\in N} \operatorname{tr}(u_g) \operatorname{Tr}(A^{(g)})\\
&=\operatorname{Tr}(A^{(e)})
\end{align*}
In our case $A^{(e)}$ is a $D\times D$ matrix defined by restricting $A$ to $D$. Therefore 
$$\operatorname{Tr}(A^{(e)})=\sum_{x\in D} \alpha_{x,x}^{(e)}\cdot \operatorname{rank}(q)=[G:N] \alpha_{e,e}^{(e)}.$$
The same relation passes to traces of matrices over the respective $C^*$-algebras.
In the case of the projection $P_n$ these formulas yield
$$\tau_N(P_n)=[G:N] \beta_{(2)}^n(G),$$
as claimed.
\end{proof}

Before we summarize this discussion we will observe one more fact that will allow to relate our results to L\"{u}ck's approximation theorem.
Recall that the Betti number of a group $G$ is the number $\beta^n(G)=\dim_{\mathbb{C}} H^n(G,\mathbb{C})$.
\begin{lemma}\label{lemma: regular betti numbers}
For a finite index subgroup $N\subseteq G$ we have
$$\dim_{\mathbb{C}} H^n(G,\ell_2(G/N))=\beta^n(N).$$
\end{lemma}
\begin{proof}
Since $N$ is of finite index in $G$ we have
$$\operatorname{CoInd}_N^G\mathbb{C} =\operatorname{Ind}_N^G\mathbb{C} =\ell_2(G/N),$$
see e.g. \cite[Proposition 5.9]{brown-book}.
Applying Shapiro's lemma \cite[Proposition 6.2]{brown-book} we obtain
$$H^n(G,\ell_2(G/N)) \simeq H^n(G,\operatorname{CoInd}_N^G\mathbb{C})\simeq H^n(N,\mathbb{C}).$$
\end{proof}

We are now in the position to formulate the main theorem of this section.

\setcounter{aaa}{\value{theorem}}
\setcounter{theorem}{1}
\begin{theorem}
Let $G$ be an exact, residually finite group of type $F_{n+1}$ and let $\{ N_i\}$ and $\mathcal{N}$ be as above.
Assume that  $\Delta_n\in  \mathbb{M}_{k_n}(C^*_{\mathcal{N}}(G))$ has a spectral gap
and that the coarse Baum-Connes assembly map 
 $KX_0(Y)\to K_0(C^*(Y))$
for the box space 
$Y=\coprod G/N_i$ of $G$
is surjective. Then
$$ \beta_{(2)}^n(N_i)= \beta^n(N_i)$$ 
for all but finitely many $i$.
\end{theorem}

\setcounter{theorem}{\value{aaa}}

\begin{proof}
The claim follows from the explicit computation of the values of both traces in  
proposition \ref{proposition: lower trace} and 
 lemmas  \ref{lemma : computation of the upper trace of P_n} and \ref{lemma: regular betti numbers}.
\end{proof}

Note that the Theorem \ref{theorem: cbcc and l2betti numbers} 
provides a strenghtening of L\"{u}ck's approximation theorem \cite{luck-gafa} in the case described by the above theorem.
Indeed, we can rewrite the conclusion of Theorem \ref{theorem: cbcc and l2betti numbers} as vanishing of 
$$[G:N_i]\left(\beta_{(2)}^n(G) - \dfrac{\beta^n(N_i)}{[G:N_i]} \right) = 0$$
for all but finitely many $i$.
Compare this with \cite[Theorem 5.1]{luck-approx-survey}, where examples with slow speed of convergence have been constructed.
The speed of convergence of Betti numbers of finite quotients to the $\ell_2$-Betti number of a residually finite group
was also studied in \cite{clair-whyte}, however the techniques used there are different. In our case both the assumptions and the conclusions are stronger. 

\section{Final remarks}
\begin{question} 
Is there a spectral gap characterization of the existence of higher Kazhdan projections in $C^*_{\mathcal{N}}(G)$?
\end{question}

As mentioned earlier, Kazhdan's property $(T)$ for $G$ is characterized by either the vanishing of first cohomology of $G$ with every unitary coefficients,
or equivalently, by the fact that the first cohomology of $G$ with any unitary coefficients is always reduced.
Related higher-dimensional generalizations of property $(T)$ were discussed in \cite{bader-nowak-deformations}, see also \cite{dechiffre-etal}.
As pointed out in \cite{bader-nowak}, the generalizations of these two conditions to higher degrees 
are not equivalent. 
The existence of higher Kazhdan projections is related to the property that cohomology is reduced and to the existence of gaps in the spectrum
of the Laplacian rather than to vanishing of cohomology. Indeed, in our reasoning it is crucial that 
cohomology does not vanish.  
It would be interesting to determine if the existence of higher Kazhdan projections can be viewed as a higher-dimensional rigidity property. 

\begin{remark}[Higher Kazhdan projections and $K$-amenability]\normalfont
Clearly, if for a particular $G$ we have at the same time that $\beta_{(2)}^n(G)=0$ and 
$\pi(\Delta_n)$ has a nontrivial kernel and a spectral gap for some $\pi\neq\lambda$,
then $G$ cannot be amenable. 

Consider the map 
$$K_*(C^*_{\max}(G))\to K_*(C^*_r(G)).$$
If $[p_n]\neq 0$ in the former, but $[p_n^r]=0$ in the latter then  the map above cannot be an isomorphism.
In other words, if $\beta_{(2)}^n(G)=0$ and we could ensure condition that the class of $[p_n]$ is not 0, then the group 
$G$ would not be $K$-amenable, as this last condition forces the above map in $K$-theory to be an isomorphism.
\end{remark}

\begin{remark}[Ghost projections]\normalfont
We can extend the notion of a ghost operator to matrix algebras over the Roe algebra by defining an element $T\in \mathbb{M}_n(C^*(Y))$ to
be a ghost if $T_{i,j}\in C^*(Y)$ is a ghost for every $1\le i,j\le n$. It can be shown that the kernel of the lifting map $\phi^{(n)}$ consists precisely of 
ghost operators.  See \cite[Corollary 13.3.14]{willett-yu-book}. This observation provides a new cohomological tool to construct ghost projections 
in the case when an $\ell_2$-Betti number of $G$ vanishes. The problem of constructing new examples 
of ghost projection was posed by Willet and Yu \cite{willett-yu1}.
\end{remark}

 \Addresses
 
\end{document}